\documentclass[11pt]{amsart}

\usepackage{amsmath,amssymb,amsthm}
\input xy
\xyoption{all}
\newtheorem{theorem}{Theorem}[section]

\newtheorem{lemm}[theorem]{Lemma}
\newtheorem{prop}[theorem]{Proposition}

\theoremstyle{definition}
\newtheorem{defi}[theorem]{Definition}

\newtheorem{exam}[theorem]{Example}

\newcommand{\PP}{\mathcal{P}}

\newcommand{\functor}[1]{{\overline{#1}}}
\newcommand{\TT}{\operatorname{\mathcal T}\nolimits}

\newcommand{\Fac}{\mathsf{Fac}\hspace{.01in}}

\newcommand{\fftors}{\mbox{\rm f-tors}\hspace{.01in}}
\newcommand{\sttilt}{\mbox{\rm s$\tau$-tilt}\hspace{.01in}}
\newcommand{\sttiltm}{\mbox{\rm s$\tau^-$-tilt}\hspace{.01in}}

\newcommand{\twosilt}{\mbox{\rm 2-silt}\hspace{.01in}}
\newcommand{\ctilt}{\mbox{\rm c-tilt}\hspace{.01in}}

\newcommand{\ffsttors}{\mbox{\rm $\nu$-f-tors}\hspace{.01in}}
\newcommand{\ssttilt}{\mbox{\rm $\nu$-s$\tau$-tilt}\hspace{.01in}}
\newcommand{\twotilt}{\mbox{\rm 2-tilt}\hspace{.01in}}
\newcommand{\sctilt}{\mbox{\rm self-c-tilt}\hspace{.01in}}

\newcommand{\C}{\mathcal{C}}

\newcommand{\thick}{\mathsf{thick}} 
\newcommand{\KKb}{\mathsf{K}^{\rm b}}
\newcommand{\La}{\Lambda}

\newcommand{\add}{\operatorname{add}\nolimits}
\newcommand{\proj}{\operatorname{proj}\nolimits}
\newcommand{\inj}{\operatorname{inj}\nolimits}

\newcommand{\Hom}{\operatorname{Hom}\nolimits}
\newcommand{\End}{\operatorname{End}\nolimits}
\newcommand{\Ext}{\operatorname{Ext}\nolimits}

\newcommand{\Tr}{\operatorname{Tr}\nolimits}
\newcommand{\RHom}{\mathbf{R}\strut\kern-.2em\operatorname{Hom}\nolimits}
\renewcommand{\mod}{\operatorname{mod}\nolimits}
\newcommand{\op}{\mathop{\mathrm{op}}\nolimits}

\begin{document}
\title{$\nu$-stable $\tau$-tilting modules}
\author{Yuya Mizuno}
\thanks{2000 { Mathematics Subject Classification.} 16G10}
\thanks{{ Key words and phrases.} selfinjective algebra, $\tau$-tilting, tilting complex, cluster tilting}
\thanks{The author is supported by Grant-in-Aid
for JSPS Fellowships No.23.5593.}
\address{Graduate School of Mathematics\\ Nagoya University\\ Frocho\\ Chikusaku\\ Nagoya\\ 464-8602\\ Japan}
\email{yuya.mizuno@math.nagoya-u.ac.jp}
%\date{\today}

\begin{abstract}
Inspired by $\tau$-tilting theory \cite{AIR}, 
we introduce the notion of $\nu$-stable support $\tau$-tilting modules. 
For any finite dimensional selfinjective algebra $\La$, we give bijections between two-term tilting complexes in $\KKb(\proj\La)$, $\nu$-stable support $\tau$-tilting $\La$-modules and $\nu$-stable functorially finite torsion classes in $\mod\La$.  
Moreover, these objects correspond bijectively to selfinjective cluster tilting objects in $\C$ if $\La$ is a 2-CY tilted algebra associated with a Hom-finite 2-CY triangulated category $\C$. 
We also study some properties of support $\tau$-tilting modules over 2-CY tilted algebras and we give a necessary condition such that algebras are 2-CY tilted in terms of support $\tau$-tilting modules.
\end{abstract}

\maketitle
%\newpage
%\tableofcontents
%\newpage
\section{Introduction}
Derived categories are nowadays considered as an essential tool in the study of many
branches of mathematics. In the representation theory of algebras, derived equivalences of
algebras have been one of the central themes and extensively investigated. It is well-known that tilting complexes give derived equivalences \cite{R}.
The most fundamental class of tilting complexes are tilting modules. 
In the case of selfinjective algebras, however, tilting modules are only projective modules, 
so that the next meaningful class are two-term tilting complexes. 

One of the aim of this paper is, for selfinjective algebras, to give a one-to-one correspondence between two-term tilting complexes and certain nice class of modules. 
By this bijection, we can obtain all two-term tilting complexes from those modules, which can be easily calculated. 
For this purpose, we use $\tau$-tilting theory introduced by Adachi-Iyama-Reiten. 

In \cite{AIR}, the authors introduced the notion of support $\tau$-tilting modules. In particular, they  gave a bijection between support $\tau$-tilting modules over a finite dimensional algebra $\La$ and two-term silting complexes in $\KKb(\proj\La)$, which are generalization of tilting complexes and play significant roles in the study of $t$-structures and mutation theory. 
Unfortunately, silting complexes do not give derived equivalences in general. 
Therefore, from the viewpoint of derived equivalences, it is reasonable to ask which support $\tau$-tilting modules correspond to two-term tilting complexes. 

In this paper, we give a complete answer to the question for selfinjective algebras by introducing $\nu$-stable support $\tau$-tilting modules (Definition \ref{selfinjective}).
Moreover, inspired by results of \cite{AIR}, we extend the bijection to functorially finite  torsion classes and cluster tilting objects.

Our main result is the following theorem.

\begin{theorem}[Theorems \ref{selfinjective-tilt}, \ref{stor-sspt}, \ref{tilt-selfct} and \ref{selfinjective-selfct}]
Let $\Lambda$ be a finite dimensional selfinjective algebra. We have bijections between
\begin{itemize}
\item[(a)] the set $\twotilt\La$ of isomorphism classes of basic two-term tilting complexes in $\KKb(\proj\La)$,
\item[(b)] the set $\ssttilt\La$ of isomorphism classes of basic $\nu$-stable support $\tau$-tilting $\La$-modules,
\item[(c)] the set $\ffsttors\La$ of $\nu$-stable functorially finite torsion classes in $\mod\Lambda$,
\item[(d)] the set $\sctilt\C$ of isomorphism classes of basic selfinjective cluster tilting objects in a 2-CY triangulated category $\C$ if $\Lambda$ is an associated 2-CY tilted algebra to $\C$.
\end{itemize}
\end{theorem}

In particular, by the correspondence of (a) and (b), we can obtain all two-term tilting complexes from $\nu$-stable support $\tau$-tilting modules, which are given by simple calculations in the module category. 
Furthermore, we show that support $\tau$-tilting modules have particularly nice properties over 2-CY tilted algebras. These properties also provide a necessary condition such that algebras are 2-CY tilted.
%Using $\tau$-tilting modules, we give a necessary condition such that algebras are 2-CY tilted.
%some relationships between support $\tau$-tilting modules and 2-CY tilted algebras give a necessary condition such that algebras are 2-CY tilted. 
%we show some relationships between support $\tau$-tilting modules and 2-CY tilted algebras give a necessary condition such that algebras are 2-CY tilted.
%we showed that support $\tau$-tilting modules coincide with support $\tau^-$-tilting modules over 2-CY tilted algebras (Theorem \ref{tau=tau-}). 
%In this case, $\nu$-stable support $\tau$-tilting modules have a particularly nice property (Proposition \ref{2-CY}). 

\subsection*{Notations}
Let $K$ be an algebraically closed field and we denote by $D:=\Hom_K(-,K)$. 
By a finite dimensional algebra $\La$, we mean a basic finite dimensional algebra over $K$. 
All modules are right modules. We denote by $\mod\Lambda$ the category of finitely generated $\Lambda$-modules, by $\proj\La$ the category of finitely generated projective $\La$-modules, by $\inj\La$ the category of finitely generated injective $\La$-modules and 
by $\KKb(\proj\La)$ the homotopy category of bounded complexes of $\proj\La$. We denote  by $\add M$ the subcategory of $\mod\Lambda$ consisting of direct summands of finite direct sums of copies of $M$.
The composition $gf$ means first $f$, then $g$. 
For $X\in\mod\La$, we denote by $\Fac X$ the subcategory of $\mod\La$ consisting of all objects which are factor modules of finite direct sums of copies of $X$.

\section{Preliminaries}
%%%%%%%%%%%%%%%%%%%%%%%%%%%%%%%%%%%%%%%%%%%%%%%%%%%%%%%%%%%%%%%%%%%%%%%%%%%%%%%%%%%% 

In this section, we recall some definitions and results. 
Throughout this section, let $\La$ be a finite dimensional algebra.

\subsection{Support $\tau$-tilting modules}
We denote the \emph{AR translations} by  
\begin{eqnarray*}
\tau=D\Tr:\underline{\mod}\Lambda\to\overline{\mod}\Lambda\ \ \ \mbox{and}\ \ \ \tau^{-1}=\Tr D:\overline{\mod}\Lambda\to\underline{\mod}\Lambda.
\end{eqnarray*}
We refer to \cite{ARS} for the functors $\Tr:\underline{\mod}\Lambda\to\underline{\mod}\Lambda^{\op}$. 
Then we recall the definition of support $\tau$-tilting modules \cite{AIR}.

Let $(X,P)$ be a pair with $X\in\mod\Lambda$ and $P\in\proj\Lambda$.
\begin{itemize}
\item[(a)] We call $X$ in $\mod\Lambda$ {\it $\tau$-rigid} if $\Hom_{\Lambda}(X,\tau X)=0$. We call $(X,P)$ a {\it $\tau$-rigid} pair if $X$ is $\tau$-rigid and $\Hom_\Lambda(P,X)=0$. 
\item[(b)] We call $X$ in $\mod\Lambda$ {\it $\tau$-tilting} if $X$ is $\tau$-rigid and $|X|=|\Lambda|$, where $|X|$ denotes the number of nonisomorphic indecomposable direct summands of $X$. 
\item[(c)]We call $X$ in $\mod\Lambda$ {\it support $\tau$-tilting} if 
there exists an idempotent $e$ of $\La$ such that $X$ is a $\tau$-tilting $(\Lambda/\langle e\rangle)$-module. 
We call $(X,P)$ a {\it support $\tau$-tilting} pair if $(X,P)$ is $\tau$-rigid and $|X|+|P|=|\Lambda|$.
\end{itemize}
We say that $(X,P)$ is basic if $X$ and $P$ are basic.

By \cite[Proposition 2.3]{AIR}, 
$(X,P)$ is a $\tau$-rigid pair for $\Lambda$ if and only if
$X$ is a $\tau$-rigid $(\Lambda/\langle e\rangle)$-module, where $e$ is an idempotent of $\Lambda$ such that $\add P=\add e\Lambda $. 
Moreover, if $(X,P)$ and $(X,P')$ are support $\tau$-tilting pairs for $\La$, then we have $\add P=\add P'$. Thus, a basic support $\tau$-tilting module $X$ gives a basic support $\tau$-tilting pair $(X,P)$ uniquely. 
We denote by $\sttilt\La$ the set of isomorphism classes of basic support $\tau$-tilting pairs for $\La$.

\subsection{Torsion classes}
We call a full subcategory $\TT$ of $\mod\Lambda$ \emph{torsion class} if it is closed under factor modules and extensions. 
We say that $X\in\TT$ is \emph{$\Ext$-projective} if $\Ext^1_\Lambda(X,\TT)=0$ and  denote by $P(\TT)$ the direct sum of one copy of each of the indecomposable $\Ext$-projective objects in $\TT$ up to isomorphism. 
We denote by $\fftors\La$ the set of functorially finite torsion classes in $\mod\La$.

\subsection{Silting complexes}
We recall the definition of silting complexes \cite{AI,BRT,KV}. 

We call a complex $T\in\KKb(\proj\La)$ \emph{silting} (respectively, \emph{tilting})  
if $\Hom_{\KKb(\proj\La)}(T,T[i])=0$ for any positive integer $i>0$ (for any integer $i\neq0$) and satisfies $\KKb(\proj\La)=\thick{T}$, where $\thick T$ denote by the smallest thick subcategory of $\KKb(\proj\La)$ containing $T$. 

We call a complex $P=(P^{i},d^i)$ in $\KKb(\proj\La)$ \emph{two-term} if $P^{i}=0$ for all $i\neq 0,-1$. 
We denote by $\twosilt\Lambda$ (respectively, $\twotilt\Lambda$) the set of isomorphism classes of basic two-term silting (respectively, tilting) complexes in $\KKb(\proj\La)$.

\subsection{Cluster tilting objects}

Let $\C$ be a $K$-linear Hom-finite Krull-Schmidt triangulated category. 
Assume that $\C$ is \emph{2-Calabi-Yau} (\emph{2-CY} for short) i.e. there exists a functorial isomorphism $\Hom_{\C}(X,Y)\cong D\Hom_{\C}(Y,X[2])$. We call $T$ in $\C$ \emph{cluster tilting} if $\add T =\{ X\in \C \ |\ \Hom_{\C}(T, X[1])=0 \}$. These categories appear in the study of cluster category and cluster tilting objects play central roles in the categories \cite{BMRRT}. We denote by $\ctilt\C$ the set of isomorphism classes of basic cluster tilting objects in $\C$.

We recall the following useful results.

\begin{lemm}\label{CT eq}\cite{KR}
Let $T\in\C$ be an object and $\La:=\End_\C(T)$.
The functor $\overline{(-)}:=\Hom_\C(T,-)$ induces an equivalence of categories between $\add T$ (respectively, $\add T[2]$) and $\proj \Lambda$ (respectively, $\inj \Lambda$).
Moreover, we have an isomorphism $\overline{(-)}\circ[2]\cong\nu\circ\overline{(-)}:\add T\to\inj \Lambda$ of functors, where $\nu:= D\Hom_{\La}(-,\La)$.
%for any $T'\in\add T$, we have$$\overline{(T'[2])}\cong\nu\overline{(T')}.$$ 
\end{lemm}

\begin{theorem}\cite{BMR,KR}\label{equivalence}
Let $T$ be a cluster tilting object and $\La:=\End_\C(T)$. 
There is an equivalence of categories
$$\Hom_\C(T,-):\C/[T[1]]\to\mod\La,$$
where $[T[1]]$ is the ideal of $\C$ consisting of morphisms which factor through $\add T[1]$.
\end{theorem}

\subsection{Bijections of \cite{AIR}}
Let $\La$ be a finite dimensional algebra.  
We recall results given by \cite{AIR}. For more details, we refer to the original paper. We also refer to independent works \cite{Ab,HKM,S} for a connection with torsion theory and tilting complexes.

We prepare notations. 
For a $\La$-module $X\in\mod\La$, take a minimal projective presentation
$$\xymatrix@C30pt@R30pt{P_X^1\ar[r]^g& P_X^0\ar[r]^{}&X\ar[r]& 0.}$$
We denote by $P_X:=({P_X^1}\overset{g}{\to} {P_X^0})\in\KKb(\proj\La)$. 

On the other hand, let $\C$ be a $K$-linear Hom-finite Krull-Schmidt 2-CY triangulated category and $T\in\C$ be a cluster tilting object. For an object $M\in\C$, we can take a triangle 
$$\xymatrix@C30pt@R30pt{T^1_M\ar[r]^{g}& T^0_M\ar[r]^{f}&M\ar[r]& T^1_M[1]}$$
where $T^0_M,T^1_M\in\add T$ and $f$ is a minimal right $(\add T)$-approximation.
We denote by $T_M:=({T^1_M}\overset{g}{\to} {T^0_M})$. 

Then we have the following results.

\begin{theorem}\cite{AIR}\label{bijections}
Let $\Lambda$ be a finite dimensional algebra.
\begin{itemize}
\item[(a)]There exists a bijection
\[\Phi:\sttilt\Lambda\longrightarrow\twosilt\Lambda,\ \ \ (X,P)\mapsto \Phi(X,P):=P_X\oplus P[1].\]

\item[(b)]There exists a bijection
\[\fftors\Lambda\longleftrightarrow\sttilt\Lambda\]
given by $\fftors\Lambda\ni\TT\mapsto P(\TT)\in\sttilt\Lambda$ and $\sttilt\La\ni X\mapsto\Fac X\in\fftors\Lambda$.
\end{itemize}

Moreover, let $\C$ be a $K$-linear Hom-finite Krull-Schmidt 2-CY triangulated category and $T\in\C$ a cluster tilting object. Assume that $\La=\End_\C(T)$. 
\begin{itemize}

\item[(c)]There exists a bijection
\[\Theta:\ctilt\C\longrightarrow\twosilt\Lambda,\ \ \ (M'\oplus M'')\mapsto\Theta(M'\oplus M''):=\overline{T_{M'}}\oplus \overline{M''[-1]}[1],\]
where $M''$ is a maximal direct summand of $M$ which belongs to $\add T[1]$.

\item[(d)]There exists a bijection
\[\Psi:\ctilt\C\longrightarrow\sttilt\Lambda,\ \ \ (M'\oplus M'')\mapsto\Psi(M'\oplus M''):=(\overline{M'}, \overline{M''[-1]}),\] where $M''$ is a maximal direct summand of $M$ which belongs to $\add T[1]$.

\end{itemize}
\end{theorem}

%%%%%%%%%%%%%%%%%%%%%%%%%%%%%%%%%%%%%%%%%%%%%%%%%%%%%%%%%%%%%%%%%%%%%%%%%%%%%%%%%%%%%%%%%%%%
\section{$\nu$-stable support $\tau$-tilting $\La$-modules}
Throughout this subsection, let $\La$ be a finite dimensional selfinjective algebra and we denote the Nakayama functor by $\nu:= D\Hom_{\La}(-,\La)$.

\subsection{Some definitions}
In this subsection, we introduce the notion of $\nu$-stable support $\tau$-tilting modules, $\nu$-stable torsion classes and selfinjective cluster tilting objects. 
 
The following notion is the main subject in this paper.
\begin{defi}\label{selfinjective}
Let $X$ (respectively, $(X,P)$) be a support $\tau$-tilting $\La$-module (support $\tau$-tilting pair for $\La$).
We call $X$ (respectively, $(X,P)$) \emph{$\nu$-stable} if $\nu X\cong X$.
%(as a $\La$-module).
\end{defi}

We denote by $\ssttilt\La$ the set of isomorphism classes of basic $\nu$-stable support $\tau$-tilting pairs for $\La$. 
Note that any support $\tau$-tilting $\La$-module is clearly $\nu$-stable if $\La$ is symmetric. 

Moreover, we call a torsion class $\TT$ \emph{$\nu$-stable} if $\nu(\TT)=\TT$, and  
we denote by $\ffsttors\La$ the set of $\nu$-stable functorially finite torsion classes in $\mod\La$.

Let $\C$ be a $K$-linear Hom-finite Krull-Schmidt 2-CY triangulated category. 
We call a cluster tilting object $X$ \emph{selfinjective} if 
$\End_\C(X)$ is selfinjective and we denote by $\sctilt\C$ the set of isomorphism classes of basic selfinjective cluster tilting objects in $\C$. 

We give the following equivalent conditions of selfinjective cluster tilting objects.

\begin{prop}\cite[Proposition 3.6.]{IO}\label{self ct}
Let $X\in\C$ be a cluster tilting object in $\C$. Then the following conditions are equivalent.
\begin{itemize}
\item[(a)]$X$ is selfinjective.
\item[(b)]$X\cong X[2]$.
\item[(c)]$\Hom_{\C}(X[1],X)=0$.
\end{itemize}
\end{prop}

Furthermore, we recall a characterization such that silting complexes over a selfinjective algebra become tilting complexes as follows.

\begin{theorem}\label{tilt eq}\cite[Theorem 2.1]{AR}\cite[Theorem A.4]{Ai}
Let $\La$ be a finite dimensional selfinjective algebra and $P$ be a basic two-term silting complex in $\KKb(\proj\La)$. 
Then the following are equivalent.
\begin{itemize}
\item[(a)] $P$ is a tilting complex.
\item[(b)] $P\cong\nu P$ in $\KKb(\proj\La)$.
\item[(c)]$\Hom_{\KKb(\proj\La)}(P[1],P)=0$.
\end{itemize}
\end{theorem}

\subsection{Connection with two-term tilting complexes}
In this subsection, we will show that $\nu$-stable support $\tau$-tilting modules correspond bijectively to two-term tilting complexes in $\KKb(\proj\La)$. 
While two-term tilting complexes are defined in $\KKb(\proj\La)$, 
$\nu$-stable support $\tau$-tilting modules are defined in $\mod\La$ and, therefore, easy to calculate.
%By our result, we can provide all two-term tilting complexes in terms of $\nu$-stable support $\tau$-tilting modules.  
%Moreover we show that they correspond bijectively to $\nu$-stable functorially finite torsion classes in $\mod\La$. 

Let us start with the following easy observation, which give a bijection between $\nu$-stable functorially finite  torsion classes in $\mod\La$ and $\nu$-stable support $\tau$-tilting $\La$-modules.

\begin{theorem}\label{stor-sspt}
The bijection of Theorem \ref{bijections} (b) induces a bijection 
$$\ffsttors\La\longleftrightarrow\ssttilt\La.$$

\end{theorem}

\begin{proof}
%Note that the Nakayama functor $\nu:\mod\La\to\mod\La$ is an equivalence.
Let $X$ be a basic $\nu$-stable support $\tau$-tilting $\Lambda$-module. 
Then it is clear that $\Fac X\cong\nu(\Fac X)$.
Conversely, let $\TT$ be a $\nu$-stable functorially finite torsion class in $\mod\Lambda$.
Then we have $\Ext_\La^1(-,\nu(\TT))\cong\Ext_\La^1(\nu^-(-),\TT)$.
Since we have $\TT=\nu(\TT)$, it is easy to obtain $P(\TT)=P(\nu(\TT))=\nu(P(\TT))$.
\end{proof}

Next, we will prove the following result. 

\begin{theorem}\label{selfinjective-tilt}
The bijection of Theorem \ref{bijections} (a) induces a bijection 
$$\ssttilt\La\longleftrightarrow\twotilt\La.$$
\end{theorem}

For a proof, we give some lemmas.

\begin{lemm}\label{nu tau-silt}
Take $X\in\mod\La$. The following are equivalent.
\begin{itemize}
\item[(a)]$X\cong\nu X$ in $\mod\La$.
\item[(b)]$P_X\cong\nu P_X$ in $\KKb(\proj\La)$.
\end{itemize}
\end{lemm}

\begin{proof}
Take a minimal projective presentation of $X$
$$\xymatrix@C30pt@R30pt{P_X^1\ar[r]& P_X^0\ar[r]^{}&X\ar[r]& 0.}$$
Then applying the functor $\nu$, 
we have the following exact sequence
$$\xymatrix@C30pt@R30pt{\nu P_X^1\ar[r]&\nu P_X^1\ar[r]^{}&\nu X\ar[r]& 0.}$$
Since $\La$ is selfinjective, this is a minimal projective presentation of $\nu X$.
Thus, we have $X\cong\nu X$ if and only if $P_X\cong\nu P_X$ in $\KKb(\proj\La)$.
\end{proof}

The following lemma is useful.

\begin{lemm}\label{nu tau e}
Let $(X,P)$ be a basic $\nu$-stable support $\tau$-tilting pair for $\La$.
Then we have $P\cong\nu P$.% (as $\La$-modules).
\end{lemm}

\begin{proof}
Let $e$ be an idempotent of $\La$ satisfying $\add e\La=\add P$. 
Since $\nu P$ is projective and $X$ is sincere as a $(\La/\langle e\rangle)$-module, 
it is enough to show that $\Hom_\La(\nu P,X)=0$. 
Then, by $X\cong\nu X$, we have 
$\Hom_\La(\nu P,X)\cong\Hom_\La(\nu P,\nu X)\cong\Hom_\La(P,X)=0.$
\end{proof}

Moreover, we give the following observation.

\begin{lemm}\label{silt split}
Let $P:=P'\oplus P''$ be a basic two-term silting complex of $\KKb(\proj\La)$ such that $P''$ is a maximal direct summand of $P$ which belongs to $\add\La[1]$.
If $\nu P\cong P$, then we have $\nu P'\cong P'$ and $\nu P''\cong P''$.
\end{lemm}

\begin{proof}
Since $\La$ is selfinjective, it is obvious.
\end{proof}

Then we give a proof of Theorem \ref{selfinjective-tilt}. 

\begin{proof}[Proof of Theorem \ref{selfinjective-tilt}]
Let $(X,P)$ be a basic $\nu$-stable support $\tau$-tilting pair for $\La$. 
By Lemmas \ref{nu tau-silt} and \ref{nu tau e}, 
we have $P_X\cong\nu P_X$ and $P\cong\nu P$. 
Thus $\nu\Phi(X,P)\cong\Phi(X,P)$ holds. 
Hence by Theorem \ref{tilt eq}, $\Phi(X,P)=P_X\oplus {P}[1]$ is a two-term tilting complex in $\KKb(\proj\La)$. 

Conversely, let $P:=P'\oplus P''$ be a basic two-term tilting complex of $\KKb(\proj\La)$ such that $P''$ is a maximal direct summand of $P$ which belongs to $\add\La[1]$. 
Since $\nu P\cong P$ by Theorem \ref{tilt eq}, we have $\nu P'\cong P'$ by Lemma \ref{silt split}. 
Then, by Lemma \ref{nu tau-silt}, we have $H^0(P')\cong\nu(H^0(P'))$.
Thus, through the bijection Theorem \ref{bijections} (a), 
it gives a $\nu$-stable support $\tau$-tilting $\La$-module.
\end{proof}

Now we will see an example. 

\begin{exam}\label{nakayama}
Let $\La=KQ/I$ be the algebra given by the following quiver $Q$ 
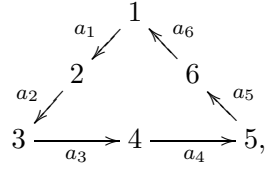
$$\xymatrix@C10pt@R10pt{
 & & 1 \ar[ld]_{a_1}& & \\
 & 2  \ar[ld]_{a_2} & & 6\ar[lu]_{a_6} & \\
3 \ar[rr]_{a_3}& & 4 \ar[rr]_{a_4}  & & 5,\ar[lu]_{a_5}}$$ 
with $I=R^4$, where $R$ denotes the arrow ideal of $KQ$.

%In this case, there are -- basic support $\tau$-tilting pairs \cite{Ad}, and there are 16 basic $\nu$-stable support $\tau$-tilting pairs. 
In this case, all basic $\nu$-stable support $\tau$-tilting pairs for $\La$ are given as follows. 
\begin{eqnarray*}
(\La,\ 0),\\
(S_1\oplus S_4\oplus P_1\oplus P_4\oplus P_3\oplus P_6,\ 0),
(S_2\oplus S_5\oplus P_2\oplus P_5\oplus P_1\oplus P_4,\ 0),\\
(S_3\oplus S_6\oplus P_3\oplus P_6\oplus P_2\oplus P_5,\ 0),
(S_1\oplus S_4\oplus \begin{smallmatrix}
3\\
4
\end{smallmatrix}\oplus
\begin{smallmatrix}
6\\
1
\end{smallmatrix}\oplus P_3\oplus P_6,\ 0),\\
(S_2\oplus S_5\oplus \begin{smallmatrix}
1\\
2
\end{smallmatrix}\oplus\begin{smallmatrix}
4\\
5
\end{smallmatrix}\oplus P_1\oplus P_4,\ 0),
(S_3\oplus S_6\oplus\begin{smallmatrix}
2\\
3
\end{smallmatrix}\oplus\begin{smallmatrix}
5\\
6
\end{smallmatrix} \oplus P_2\oplus P_5,\ 0),\\
(S_1\oplus S_4\oplus \begin{smallmatrix}
1\\
2
\end{smallmatrix}\oplus
\begin{smallmatrix}
4\\
5
\end{smallmatrix}\oplus P_1\oplus P_4,\ 0),
(S_2\oplus S_5\oplus \begin{smallmatrix}
2\\
3
\end{smallmatrix}\oplus\begin{smallmatrix}
5\\
6
\end{smallmatrix}\oplus P_2\oplus P_5,\ 0),\\
(S_3\oplus S_6\oplus\begin{smallmatrix}
3\\
4
\end{smallmatrix}\oplus\begin{smallmatrix}
6\\
1
\end{smallmatrix} \oplus P_3\oplus P_6,\ 0),\\
(S_1\oplus S_4\oplus\begin{smallmatrix}
3\\
4
\end{smallmatrix}\oplus\begin{smallmatrix}
6\\
1
\end{smallmatrix},\ P_2\oplus P_5),
(S_2\oplus S_5\oplus\begin{smallmatrix}
1\\
2
\end{smallmatrix}\oplus\begin{smallmatrix}
4\\
5
\end{smallmatrix} ,\ P_3\oplus P_6),\\
(S_3\oplus S_6\oplus\begin{smallmatrix}
2\\
3
\end{smallmatrix}\oplus
\begin{smallmatrix}
5\\
6
\end{smallmatrix} ,\ P_1\oplus P_4),
(S_6\oplus S_3\oplus\begin{smallmatrix}
3\\
4
\end{smallmatrix}\oplus\begin{smallmatrix}
6\\
1
\end{smallmatrix} ,\ P_2\oplus P_5),\\
(S_1\oplus S_4\oplus\begin{smallmatrix}
1\\
2
\end{smallmatrix}\oplus\begin{smallmatrix}
4\\
5
\end{smallmatrix} ,\ P_3\oplus P_6),
(S_2\oplus S_5\oplus\begin{smallmatrix}
2\\
3
\end{smallmatrix}\oplus
\begin{smallmatrix}
5\\
6
\end{smallmatrix},\ P_1\oplus P_4),\\
(S_1\oplus S_4,\ P_2\oplus P_3\oplus P_5\oplus P_6),
(S_2\oplus S_5,\ P_1\oplus P_3\oplus P_4\oplus P_6),\\
(S_3\oplus S_6,\ P_1\oplus P_2\oplus P_4\oplus P_5),
(0,\ \La).
\end{eqnarray*}

Then the map $\Phi$ of Theorem \ref{bijections} (a) gives all two-term tilting complexes of $\KKb(\proj\La)$.

For example, $\Phi(S_3\oplus S_6\oplus\begin{smallmatrix}
2\\
3
\end{smallmatrix}\oplus
\begin{smallmatrix}
5\\
6
\end{smallmatrix}, P_1\oplus P_4)$ is the following tilting complex.
$$
\left\{\begin{array}{ccc}
\stackrel{-1}{ P_4\oplus P_1\oplus P_4\oplus P_1}&\stackrel{\left(\begin{smallmatrix}a_3&0&0&0\\ 
0&a_6&0&0\\
0&0&a_2a_3&0\\
0&0&0&a_5a_6
\end{smallmatrix}\right)}{\longrightarrow}&\stackrel{0}{P_3\oplus P_6\oplus P_2\oplus P_5.}\\
\oplus&&\\
P_1\oplus P_4&&
\end{array}\right.
$$ 
Thus, we can obtain all two-term tilting complexes from $\nu$-stable support $\tau$-tilting modules. We refer to \cite{Ad} for calculations of support $\tau$-tilting modules over selfinjective Nakayama algebras. 
\end{exam}

%%%%%%%%%%%%%%%%%%%%%%%%%%%%%%%%%%%%%%%%%%%%%%%%%%%%%%%%%%%%%%%%%%%%%%%%%%%%%%%%%%%%%%%%%%%%
\section{Connection with cluster tilting objects}
In this section, we consider 2-CY tilted algebras and we will show that selfinjective cluster tilting objects correspond bijectively to two-term tilting complexes and $\nu$-stable support $\tau$-tilting modules. Note that Jacobi-finite algebras are 2-CY tilted algebras \cite{Am}. Hence, selfinjective quivers with potentials in the sense of \cite{HI} provide us with a rich source of selfinjective 2-CY tilted algebras.
 
Throughout this section, let $\C$ be a $K$-linear Hom-finite Krull-Schmidt triangulated 2-CY category with a cluster tilting object $T\in\C$. 
We assume that $\La:=\End_\C(T)$ is selfinjective (i.e $T$ is selfinjective) and let $\overline{(\ )}:=\Hom_\C(T,-)$.

The first aim of this section is to prove the following result.

\begin{theorem}\label{tilt-selfct}
The bijection of Theorem \ref{bijections} (c) induces a bijection 
$$\sctilt\C\longleftrightarrow\twotilt\La.$$
\end{theorem}

Recall that, for an object $M\in\C$, we can take a triangle 
$$\xymatrix@C30pt@R30pt{T^1_M\ar[r]^g& T^0_M\ar[r]^{f}&M\ar[r]& T^1_M[1]}$$
where $T^0_M,T^1_M\in\add T$ and $f$ is a minimal right $(\add T)$-approximation.
We denote by $T_M:=({T^1_M}\overset{g}{\to} {T^0_M})$ and we write $T_M\cong T_M[2]$ if there exist isomorphisms $a_0:T_M^0\to T_M^0[2]$ and $a_1:T_M^1\to T_M^1[2]$ 
which make the following diagram commute:

\[\xymatrix{
T_M^1\ar[r]^g\ar[d]_{a_1}&T_M^0\ar[d]^{a_0}\\
T_M^1[2]\ar[r]^{g[2]}&T_M^0[2].
}\]

We start with the following lemma.

\begin{lemm}\label{nu silt-ct}
Let $M$ be a basic object of $\C$. 
The following are equivalent.
\begin{itemize}
\item[(a)]$M\cong M[2]$ in $\C$.
\item[(b)]$\overline{T_M}\cong\nu\overline{T_M}$ in $\KKb(\proj\La)$.
\end{itemize}
\end{lemm}

\begin{proof}
Take a triangle 
$$\xymatrix@C30pt@R30pt{T^1_M\ar[r]& T^0_M\ar[r]^{f}&M\ar[r]& T^1_M[1]}$$
where $T^0_M,T^1_M\in\add T$ and $f$ is a minimal right $(\add T)$-approximation. 
Then we have the following triangle 
$$\xymatrix@C30pt@R30pt{T^1_M[2]\ar[r]& T^0_M[2]\ar[r]^{f[2]}&M[2]\ar[r]& T^1_M[3].}$$

Then it is easy to check that $M\cong M[2]$ if and only if $T_M\cong T_M[2]$ in $\C$. 
Hence we have $T_M\cong T_M[2]$ in $\C$ if and only if $\overline{T_M}\cong\nu\overline{T_M}$ in $\KKb(\proj\La)$ by Lemma \ref{CT eq}.
\end{proof}

\begin{lemm}\label{ct split}
Let $M=M'\oplus M''$ be a basic cluster tilting object of $\C$ such that $M''$ is a maximal direct summand of $M$ which belongs to $\add T[1]$.
If $M\cong M[2]$, we have $M'\cong M'[2]$ and $M''\cong M''[2]$.
\end{lemm}

\begin{proof}
It is enough to show that $M''\cong M''[2].$
We only have to show that $M''[2]\in\add T[1]$.
By the assumption, we have $M''\cong T'[1]$, where $T'\in\add T$. 
Since $\La=\End_\C(T)$ is selfinjective, we have $T\cong T[2]$ by Proposition \ref{self ct}. 
Now we let $T''= T'[2]$, where $T''\in\add T$. 
Then we have $M''[2]\cong (T'[1])[2]\cong (T'[2])[1]\cong T''[1]\in\add T[1]$. 
This completes the proof.
\end{proof}

Then we give a proof of Theorem \ref{tilt-selfct}.

\begin{proof}[Proof of Theorem \ref{tilt-selfct}]
Let $M=M'\oplus M''$ be a basic selfinjective cluster tilting object of $\C$ such that $M''$ is a  maximal direct summand of $M$ which belongs to $\add T[1]$. 
By Proposition \ref{self ct}, we have $M\cong M[2]$ and hence 
we obtain $M'\cong M'[2]$ and $M''\cong M''[2]$ from Lemma \ref{ct split}. 
Then by Lemma \ref{CT eq} and Lemma \ref{nu silt-ct}, 
we have $\overline{T_{M'}}\cong\nu (\overline{T_{M'}})$ and $\overline{M''[-1]}\cong\nu (\overline{M''[-1]})$. 
Thus $\Theta(M)$ is a two-term tilting complex by Theorem \ref{tilt eq}.

Conversely, let $P:=P'\oplus P''$ be a basic two-term tilting complex of $\KKb(\proj\La)$ such that $P''$ is a maximal direct summand of $P$ which belongs to $\add\La[1]$. 
By Theorem \ref{bijections} (c), there exists the corresponding cluster tilting object $M=M'\oplus M''$ such that $M''$ is a maximal direct summand of $M$ which belongs to $\add T[1]$ and 
$P'=\overline{T_{M'}}$ and $P''=\overline{M''[-1]}$. 
On the other hand, by Theorem \ref{tilt eq} and Lemma \ref{silt split}, 
we get $P'\cong\nu P'$ and $P''\cong\nu P''$.
Then, by Lemma \ref{CT eq} and Lemma \ref{nu silt-ct}, 
we have $M'\cong M'[2]$ and $M''\cong M''[2]$.
Thus, by Proposition \ref{self ct}, $M$ is a selfinjective cluster tilting object.
\end{proof}

By Theorems \ref{selfinjective-tilt} and \ref{tilt-selfct}, we have a bijection between 
$\ssttilt\La$ and $\sctilt\C$. 
Here, using the bijection of Theorem \ref{bijections} (d), we give a direct correspondence. 
 
\begin{theorem}\label{selfinjective-selfct}
The bijection of Theorem \ref{bijections} (d) induces a bijection 
$$\sctilt\C\longleftrightarrow\ssttilt\La.$$
\end{theorem}

\begin{proof}
Let $M=M'\oplus M''$ be a basic selfinjective cluster tilting object of $\C$ such that $M''$ is a maximal direct summand of $M$ which belongs to $\add T[1]$. 
By Proposition \ref{self ct}, we have $M\cong M[2]$ and hence 
we obtain $M'\cong M'[2]$ and $M''\cong M''[2]$ from Lemma \ref{ct split}. 
Then by Lemma \ref{CT eq} and Lemma \ref{nu silt-ct}, 
we have $\overline{M'}\cong\nu \overline{M'}$ and $\overline{M''[-1]}\cong\nu \overline{M''[-1]}$. 
Thus $\Psi(M)$ is a $\nu$-stable support $\tau$-tilting pair for $\La$.

Conversely, let $(X,P)$ be a basic $\nu$-stable support $\tau$-tilting pair for $\La$. 
By Theorem \ref{bijections} (d), there exists the corresponding cluster tilting object $M=M'\oplus M''$ such that $M''$ is a maximal direct summand of $M$ which belongs to $\add T[1]$ and 
$X=\overline{M'}$ and $P=\overline{M''[-1]}$.  
By Lemmas \ref{nu tau-silt} and \ref{nu silt-ct}, 
we have $M'\cong M'[2]$. 
On the other hand, by Lemma \ref{nu tau e}, 
we have $P\cong\nu P$ and hence $M''\cong M''[2]$ by Lemma \ref{CT eq}.
Thus we get $M\cong M[2]$ and, by Proposition \ref{self ct}, $M$ is a selfinjective cluster tilting object.
\end{proof}

\begin{exam}\label{exam2}
Let $\La=KQ/I$ be the finite dimensional algebra given by the following quiver $Q$  
$$\xymatrix@C30pt@R30pt{
1\ar[d]_{a_1}    &4\ar[l]_{a_4}  \\
2 \ar[r]_{a_2}  &  3,\ar[u]_{a_3} }$$
with $I=R^3$, where $R$ denotes the arrow ideal of $KQ$. 
In this case, $\La$ is given by the Jacobian algebra $\PP(Q,W)$ for a potential $W=a_1a_2a_3a_4$ (see \cite{DWZ}) and hence it is a 2-CY tilted algebra by \cite{Am}.

Thus, there exists a 2-CY category $\C$ and 
cluster tilting object $T\in\C$ such that 
$\End_\C(T)\cong\La.$ 
Let $T:=T_1\oplus T_2 \oplus T_3\oplus T_4$. Then, for example, $\nu$-stable support $\tau$-tilting $\La$-module $$S_1\oplus S_3\oplus P_1\oplus P_3$$ 
corresponds to the following two-term tilting complex 
$$\left\{\begin{array}{ccc}
\stackrel{-1}{ P_2\oplus P_4}&\stackrel{\left(\begin{smallmatrix} a_1&0\\
0&a_3\\
\end{smallmatrix}\right)}{\longrightarrow}&\stackrel{0}{ P_1\oplus P_3}\\
&&\oplus\\
&&P_1\oplus P_3
\end{array}\right.$$

and to the following selfinjective cluster tilting object 
$${{\rm cone}(f_1)}\oplus{{\rm cone}(f_3)}\oplus{T_1}\oplus{T_3},$$ 
where $f_i$ $(i=1,3)$ is a minimal left $(\add(T/T_i))$-approximation of $T_i$.

\end{exam}

\section{support $\tau$-tilting modules and 2-CY tilted algebras}
%%%%%%%%%%%%%%%%%%%%%%%%%%%%%%%%%%%%%%%%%%%%%%%%%%%%%%%%%%%%%%%%%
In this section, we investigate support $\tau$-tilting modules over 2-CY tilted algebras more closely and provide some nice properties. In \cite{KR}, the authors have shown that 2-CY tilted algebras are Gorenstein of dimension at most 1.  
We give different types of conditions such that algebras are 2-CY tilted in terms of support $\tau$-tilting modules.

%These results give a necessary condition that selfinjective algebra are 2-CY tilted. 
%In particular, we show that support $\tau^-$-tilting modules coincide with support $\tau^-$-tilting modules. 
%Moreover, we show that $\nu$-stable support $\tau$-tilting $\La$-modules have a particularly nice property. 

Throughout this section, let $\C$ be a $K$-linear Hom-finite Krull-Schmidt triangulated 2-CY category with a cluster tilting object $T\in\C$. We let $\La:=\End_\C(T)$ (not necessarily selfinjective) and $\overline{(\ )}:=\Hom_\C(T,-)$.

First we give the following definition, which is a dual notion of support $\tau$-tilting modules.  

\begin{defi}\label{tau-}\cite{AIR} Let $\La$ be a finite dimensional algebra. 
\begin{itemize}
\item[(a)] We call $X$ in $\mod\Lambda$ \emph{$\tau^-$-rigid} if $\Hom_\Lambda(\tau^-X,X)=0$.
\item[(b)] We call $X$ in $\mod\Lambda$ \emph{$\tau^-$-tilting} if $X$ is $\tau^-$-rigid and $|X|=|\Lambda|$. 
\item[(c)] We call $X$ in $\mod\Lambda$ \emph{support $\tau^-$-tilting} if $X$ is a $\tau^-$-tilting $(\Lambda/\langle e\rangle)$-module for some idempotent $e$ of $\Lambda$.
\end{itemize}
\end{defi}

Clearly $X$ is a $\tau^-$-rigid (respectively, $\tau^-$-tilting, support $\tau^-$-tilting) $\La$-module if and only if $DX$ is a $\tau$-rigid (respectively, $\tau$-tilting, support $\tau$-tilting) $\La^{\op}$-module. 
We denote by $\sttiltm\La$ the set of isomorphism classes of basic support $\tau^-$-tilting $\La$-modules. 

We start with the following lemma, which is an analog of \cite[Proposition 4.3]{AIR}.

\begin{lemm}\label{x=x[2]}
Let $X,Y$ be objects in $\C$. Assume that
there are no nonzero indecomposable direct summands of $T[1]$ for $X$ and $Y$.
\begin{itemize}
\item[(a)]We have $\functor{X[-1]}\cong\tau^-\functor{X}$ as $\La$-modules.
\item[(b)]We have an exact sequence
$$0\to D\Hom_{\Lambda}(\tau^-\functor{Y},\functor{X})\to\Hom_{\C}(X[-1],Y)\to\Hom_{\Lambda}(\tau^-\functor{X},\functor{Y})\to 0.$$
\item[(c)] We have $\Hom_{\C}(X[-1],X)=0$ if and only if $\Hom_\La(\tau^-\functor{X},\functor{X})=0$. 

\end{itemize}
\end{lemm}

\begin{proof}
(a)  
Since $T[1]$ is also a cluster tilting object, we can take a triangle
\begin{equation}\label{T resolution}
\xymatrix{T_1[1]\ar[r]^{g}&T_0[1] \ar[r]^{f}& X\ar[r]& T_1[2]}
\end{equation}
with a minimal left $(\add (T[1]))$-approximation $f$ and $T_0,T_1\in\add T$.
Applying $\Hom_\C(T,-)$ to \eqref{T resolution}, we have an exact sequence
\begin{equation}\label{min proj resol}
\xymatrix{0\ar[r]& \functor{X}\ar[r]^{}& \functor{T_1[2]}\ar[r]^{\functor{g[1]}}&\functor{T_0[2].}}
\end{equation}

%Since $X$ has no nonzero indecomposable direct summands of $T[1]$, it gives an injective resolution. 
Applying the inverse of Nakayama functor to \eqref{min proj resol} and $\Hom_{\C}(T,-)$ to \eqref{T resolution}, and comparing them by Lemma \ref{CT eq}, we have the following commutative diagram of exact sequences:

\[\xymatrix{
\nu^-\functor{T_1[2]}\ar[r]^{\nu^-\functor{g[1]}}\ar[d]^\wr&\nu^-\functor{T_0[2]}\ar[r]\ar[d]^\wr&\tau^-\functor{X}\ar[r]^{}&0\\
\functor{T_1}\ar[r]^{\functor{g[-1]}}&\functor{T_0}\ar[r]&\functor{X[-1]}\ar[r]^{}&\functor{T_1[1]}=0.
}\]
Thus, we have $\tau^-\functor{X}\cong\functor{X[-1]}$.

(b) We have an exact sequence
\[0\to[T[1]](X[-1],Y)\to\Hom_{\C}(X[-1],Y)\to\Hom_{\C/[T[1]]}(X[-1],Y)\to0,\]
where $[T[1]]$ is the ideal of $\C$ consisting of morphisms which factor through $\add T[1]$. By Theorem \ref{equivalence} and (a), 
we have the following functorial isomorphism
\begin{equation*}\label{X,Y[1]}
\Hom_{\C/[T[1]]}(X[-1],Y)\cong\Hom_\Lambda(\functor{X[-1]},\functor{Y}){\cong}\Hom_\Lambda(\tau^-\functor{X},\functor{Y}).
\end{equation*}
Moreover, using \cite[Lemma 3.3]{P}, we have the following functorial isomorphism 
\[[T[1]](X[-1],Y)\cong D\Hom_{\C/[T[1]]}(Y[-1],X)\cong D\Hom_\Lambda(\tau^-\functor{Y},\functor{X}).\]
Thus the assertion follows.

(c) This is immediate from (b).
\end{proof}

Then, we have the following conclusion.

\begin{theorem}\label{tau=tau-}
Let $\La$ be a 2-CY tilted algebra. Then,  
support $\tau$-tilting $\La$-modules coincide with support $\tau^-$-tilting $\La$-modules. \end{theorem}

\begin{proof}
Let $X$ be a support $\tau$-tilting $\La$-module. 
By Theorem \ref{bijections} (d), there exists the corresponding cluster tilting object $M=M'\oplus M''$ such that $M''$ is a maximal direct summand of $M$ which belongs to $\add T[1]$ and $X=\overline{M'}$. 
By \cite[Proposition 4.3]{AIR}, we have $\Hom_\La(X,\tau X)=0$ if and only if $\Hom_\C(M',M'[1])=0$. Hence we have $\Hom_\C(M'[-1],M')=0$. 
Moreover, by Lemma \ref{x=x[2]}, we have $\Hom_\C(M'[-1],M')=0$ if and only if $\Hom_\La(\tau^-X, X)=0$. 
Thus $X$ is a support $\tau^-$-tilting $\La$-module.
Conversely, assume that $X$ is a support $\tau^-$-tilting $\La$-module. 
Since $DX$ is a support $\tau$-tilting $\La^{\op}$-module and $\La^{\op}$ is clearly 2-CY tilted, we can prove similarly that $DX$ is a support $\tau^-$-tilting $\La^{\op}$-module. Thus,  
$X$ is a support $\tau$-tilting $\La$-module.
\end{proof}

%The converse implication follows easily by the fact that 
%Then $DX$ is a support $\tau$-tilting $\La^{\op}$-module. [Clearly $\La^{\op}$ is a 2-CY tilted,]Similarly, $DX$ is a support $\tau^-$-tilting $\La^{\op}$-module and hence $X$ is a support $\tau$-tilting $\La$-module.
%$\tau X=\functor{M'[1]}$. On the other hand, by Lemma \ref{x=x[2]}, we have $\tau^- X=\functor{M'[-1]}$.
%By Lemmas \ref{nu tau-silt} and \ref{nu silt-ct}, we obtain $M'\cong M'[2]$.Then, by Lemma \ref{x=x[2]},  we get $\tau^-X\cong\tau X$.

Moreover, $\nu$-stable support $\tau$-tilting modules have the following strong property if $\La$ is selfinjective.

\begin{prop}\label{2-CY}
Let $\La$ be a selfinjective 2-CY tilted algebra.
For any basic $\nu$-stable support $\tau$-tilting $\La$-module $X$, 
we have $\tau^-X\cong\tau X$ as $\La$-modules. 
\end{prop}

\begin{proof}
By Theorem \ref{selfinjective-selfct}, there exists the corresponding selfinjective cluster tilting object $M=M'\oplus M''$ such that $M''$ is a maximal direct summand of $M$ which belongs to $\add T[1]$ and $X=\overline{M'}$. 
By Lemmas \ref{nu tau-silt} and \ref{nu silt-ct}, we obtain $M'\cong M'[2]$.
Then, by \cite[Proposition 4.3]{AIR} and Lemma \ref{x=x[2]},  
we get $\tau X\cong\functor{M'[1]}\cong\functor{M'[-1]}\cong\tau^- X$.
\end{proof}

\begin{exam}
Let $\La$ be the algebra given in Example \ref{exam2}. 
Then the AR quiver is given as follows.

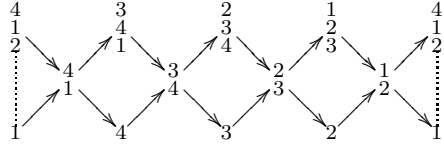
\[
\begin{xy}
(0,0) *{_{1}}="A" ,
(7,7) *{_{\begin{smallmatrix}4\\1\end{smallmatrix}}}="B" ,
(0,14) *{_{\begin{smallmatrix}4\\1\\2\end{smallmatrix}}}="BBB" ,
(14,14) *{_{\begin{smallmatrix}3\\4\\1\end{smallmatrix}}}="BB" ,
(14,0) *{_{4}}="C" ,
(21,7) *{_{\begin{smallmatrix}3\\4\end{smallmatrix}}}="D" ,
(28,14) *{_{\begin{smallmatrix}2\\3\\4\end{smallmatrix}}}="DD" ,
(28,0) *{_{3}}="E" ,
(35,7) *{_{\begin{smallmatrix}2\\3\end{smallmatrix}}}="F" ,
(42,14) *{_{\begin{smallmatrix}1\\2\\3\end{smallmatrix}}}="FF" ,
(42,0) *{_{2}}="G" ,
(49,7) *{_{\begin{smallmatrix}1\\2\end{smallmatrix}}}="H" ,
(56,14) *{_{\begin{smallmatrix}4\\1\\2\end{smallmatrix}}}="HH" ,
(56,0) *{_{1}}="I" ,

\ar@{.} "A" ; "BBB"
\ar "BBB" ; "B"
\ar "A" ; "B"
\ar "B" ; "C" 
\ar "C" ; "D"
\ar "D" ; "E"
\ar "E" ; "F"
\ar "F" ; "G"
\ar "G" ; "H"
\ar "B" ; "BB"
\ar "BB" ; "D"
\ar "D" ; "DD"
\ar "DD" ; "F"
\ar "F" ; "FF"
\ar "FF" ; "H"
\ar "H" ; "HH"
\ar "H" ; "I"
\ar@{.} "HH" ; "I"
\end{xy}
\]
Then it is easy to check that $\sttilt\La=\sttiltm\La$. 
On the other hand, for example, $X=S_1\oplus S_3\oplus P_1\oplus P_3$ is a $\nu$-stable $\tau$-tilting $\La$-module and we can see that $\tau X=\tau^-X$.
%support $\tau$-tilting modules coincide with support $\tau^-$-tilting modules.
%For example, $X=S_1\oplus S_3\oplus P_1\oplus P_3$ is a $\nu$-stable $\tau$-tilting $\La$-module and we can see that $\tau X=\tau^-X$.
\end{exam}

%\begin{cor}\label{app}Let $\Gamma$ be a finite dimensional selfinjective algebra. 
%If $\Gamma$ is 2-CY tilted, then we have $\sttilt\Gamma=\sttiltm\Gamma$, and $\tau^-X\cong\tau X$ for any $\nu$-stable support $\tau$-tilting $\Gamma$-module $X$. \end{cor}

\begin{exam}
Let $\Gamma$ be a preprojective algebra of Dynkin quiver of $A_3$ 
$$\xymatrix{1\ar[r]&2\ar[r]\ar@<-1ex>[l]&3\ar@<-1ex>[l]}.$$
Then it is known that $\Gamma$ is selfinjective. 
For example, we have a $\nu$-stable support $\tau$-tilting $\La$-module $X=P_2\oplus\begin{smallmatrix}
2\\
3
\end{smallmatrix}\oplus
\begin{smallmatrix}
2\\
1
\end{smallmatrix}$ and it is easy to check that $\tau X\ncong\tau^-X$. 
Hence, we can immediately conclude that $\Gamma$ is not 2-CY tilted.
\end{exam}

\section*{Acknowledgement}
First and foremost, the author would like to thank Osamu Iyama for his support and patient guidance. He is grateful to Kota Yamaura for his kind support and Takahide Adachi for his  valuable comments.
He thanks Aaron Chan for a correction of example \ref{nakayama}. 
The author is grateful to the anonymous referee for the valuable comments.

%%%%%%%%%%%%%%%%%%%%%%%%%%%%%%%%%%%%%%%%%%%%%%%%%%%%%%%%%%%%%%%%%%%%%%%%%%%%%%%%%%%%%%%%%%%%%%%%%%%%%%%%%%%%%%%%%%%%%%%%%%


\begin{thebibliography}{20}
\bibitem[Ab]{Ab} H.~Abe, {\it Tilting modules arising from two-term tilting complexes},
Comm. in algebra 34 (2006), no. 12, 4441--4452.
\bibitem[Ad]{Ad} T.~Adachi, \emph{$\tau$-tilting modules over Nakayama algebras}, arXiv:1309.2216.
\bibitem[AIR]{AIR} T.~Adachi, O. Iyama, I. Reiten, {\it $\tau$-tilting theory},
to appear in Compos. Math, arXiv:~1210.1036.
\bibitem[Ai]{Ai} T.~Aihara, {\it Tilting-connected selfinjective algebras},
to appear in Algebr. Represent. Theory, arXiv:~1012.3265.
\bibitem[AI]{AI} T.~Aihara, O.~Iyama, {\it Silting mutation in triangulated categories},
J. Lond. Math. Soc. 85 (2012), no. 3, 633--668.
\bibitem[AR]{AR} S. Al-Nofayee, J. Rickard, 
{\it Rigidity of tilting complexes and derived equivalence for self-injective algebras}, 
preprint. 
\bibitem[Am]{Am} C. Amiot, {\it Cluster categories for algebras of global dimension 2 and quiver with potential}, Ann. Inst. Fourier (Grenoble) 59 (2009), no. 6, 2525--2590.
\bibitem[ARS]{ARS} M. Auslander, I. Reiten, S. O. Smal\o, {\it Representation theory of artin algebras}, Cambridge studies in advanced mathematics 36, Cambridge Univ. Press 1995.
\bibitem[BMRRT]{BMRRT} A. B. Buan, R. Marsh, M. Reineke, I. Reiten, G. Todorov, {\it Tilting theory and cluster combinatorics}, Adv. Math. 204 (2006), no. 2, 572--618.
\bibitem[BMR]{BMR} A.B. Buan, R. Marsh, I. Reiten, {\it Cluster-tilted algebras}, Trans. Amer. Math. Soc. 359(2007), no. 1, 323--332.
\bibitem[BRT]{BRT} A. B. Buan, I. Reiten, H. Thomas, {\it Three kinds of mutation}, J. Algebra 339 (2011), 97--113.
\bibitem[DWZ]{DWZ} H.~Derksen, J.~Weyman, A.~Zelevinsky, 
{\it Quivers with potentials and their representations. {I}. {M}utations}, Selecta Math. (N.S.) 14 (2008), no. 1, 59--119.
\bibitem[HI]{HI} M.~Herschend, O.~Iyama, \emph{Selfinjective quivers with potential and 2-representation-finite algebras}, Compos. Math. 147 (2011), no. 6, 1885--1920. 
\bibitem[HKM]{HKM} M. Hoshino, Y. Kato, J. Miyachi, {\it On $t$-structures and torsion theories induced by compact objects}, J. Pure Appl. Algebra 167 (2002), no. 1, 15--35.
\bibitem[IO]{IO} O. Iyama, S. Oppermann, {\it Stable categories of higher preprojective algebras}, to appear in Adv. Math. arXiv:0912.3412.
\bibitem[KR]{KR} B. Keller, I. Reiten, {\it Cluster-tilted algebras are Gorenstein and stably Calabi-Yau}, Adv. Math. 211 (2007), 123--151.
\bibitem[KV]{KV} B.~Keller, D.~Vossieck, {\it Aisles in derived categories},
Deuxieme Contact Franco-Belge en Algebre (Faulx-les-Thombes, 1987).
Bull.\ Soc.\ Math.\ Belg.\ \textbf{40} (1988), 239--253.
\bibitem[P]{P} Y. Palu, {\it Cluster characters for 2-Calabi-Yau triangulated categories}, Ann. Inst. Fourier (Grenoble) 58 (2008), no. 6, 2221--2248.
\bibitem[R]{R} J. Rickard, {\it Morita theory for derived categories}, J. London Math. Soc. (2) 39 (1989), no. 3, 436--456.
\bibitem[S]{S} S. O. Smal\o, {\it Torsion theory and tilting modules}, Bull. London Math. Soc. 16 (1984), 518--522.
\end{thebibliography}
\end{document}